\documentclass{amsart}

\usepackage{amsmath, amsthm, amssymb, tikz, hyperref}
\usepackage{graphicx}
\newtheorem{thm}{Theorem}[section]

\newtheorem{lem}[thm]{Lemma}
\newtheorem{cor}[thm]{Corollary}
\newtheorem{prop}[thm]{Proposition}

\theoremstyle{definition}
\newtheorem{dfn}[thm]{Definition}

\newtheorem{qst}[thm]{Question}
\theoremstyle{remark}

\newcommand{\im}{\operatorname{im}}

\newcommand{\R}{\mathbb{R}}
\newcommand{\N}{\mathbb{N}}
\newcommand{\Z}{\mathbb{Z}}
\newcommand{\Q}{\mathbb{Q}}

\newcommand{\bb}[1]{\mathbb{{#1}}}
\newcommand{\inv}{^{-1}}

\newcommand{\dom}{\operatorname{dom}}

\newcommand{\aut}{\operatorname{Aut}}

\newcommand{\res}{\upharpoonright}

\newcommand{\stab}{\operatorname{stab}}

\newcommand{\lra}{\Leftrightarrow}

\newcommand{\s}[1]{\mathcal{{#1}}} 

\title{Fiid coloring random maps}
\author[J Hsu]{Justin Hsu}
\address{Carnegie Mellon University, Dept. of Mathematics, 5000 Forbes Ave, Pittsburgh, PA 15213}
\email{jhsu3@andrew.cmu.edu}

\author[D Sium]{Daniel Sium}
\address{Carnegie Mellon University, Dept. of Mathematics, 5000 Forbes Ave, Pittsburgh, PA 15213}
\email{dsium@andrew.cmu.edu}

\author[R Thornton]{Riley Thornton}
\address{University of Michigan, Department of Mathematics, 530 Church St, Ann Arbor, MI 48109}
\email{rileyjt@umich.edu}

\begin{document}

\begin{abstract}
    We show that the percolation components on a square grid can be $5$-colored as a factor of iid so that neighboring components get different colors. Above the critical probability, we observe that these regions can be 4-colored. Along the way, we extend the classical correspondence between fiid process and measurable labellings of the Bernoulli shift to this quotient setting, prove a general coloring result about hyperfinite pmp planar graphs, and clarify some foundational issues about how to define planarity for Borel graphs.
\end{abstract}

\maketitle

\section{Introduction}

The infamous four color theorem states that any planar map (that is, any partition of the plane into connected regions) can be colored so that neighboring regions have different colors \cite{AppelHaken1977, AppelHakenKoch1977}. In this paper, we explore measurable and local variations of this theorem. What if each region needs to decide its color independently with limited communication? What if the map is generated randomly and the coloring needs to be computed as a factor of the map? What if we want to color infinite maps definably, without the axiom of choice? 

These questions are intimately linked. We can use factor of iid processes as a simple model for local computation. Factor of iid processes for a countable transitive graph correspond to labellings of an associated Bernoulli graphing, and measurability is closely linked to definability. See the survey of Grebik and Rozhon for details in the case of coloring problems on grids \cite{gridssurvey}.

 In particular, we will try to color planar maps arising from percolation. If we start with an infinite square grid and delete each edge with probability $(1-p)$ (that is, if we consider Bernoulli bond $p$-percolation), we break the grid into connected regions. These are the regions of the map we will try to color. 

Write $\chi_{fiid}((\widetilde{\Z^2})_p)$ for the number of colors need to color our percolation maps with a factor of iid process (a precise definition is given in Section \ref{sec: perc quotient}.) Our main theorem is this:

\begin{thm}
    For $p>\frac{1}{2}$, $\chi_{fiid}((\widetilde{\Z^2})_p)\leq 4$. For $p\leq \frac12$, $\chi_{fiid}((\widetilde{\Z^2})_p)\leq 5$.
\end{thm}

Our bound in the supercritical regime follows easily from deep results about percolation on the grid and raises intriguing questions for more general graphs. Our bound in the subcritical regime follows from a much more general theorem (which we expect is optimal).

\begin{thm}
    For any locally finite planar hyperfinite pmp graph $\s G$, $\chi_{\mu}(\s G)\leq 5.$
\end{thm} 

In section \ref{sec: prelim}, we lay out the necessary background on graphs and percolation. In section \ref{sec: perc quotient}, we state our problem precisely. In section \ref{sec: supercritical}, we explain the bound in the supercritical regime and suggest some related open problems. In section \ref{sec: amenable}, we show how to use a specialized toast structure to color planar graphs. And, lastly in section \ref{sec: toast}, we explain how to construct the toast we need to get our main results.

\subsection{Acknowledgments} Thank you to Andrew Marks and Katalin Berlow for helpful conversations about the content of this paper. Much of this research was carried out as part of Carnegie Mellon University's SEMS program; we also thank Irina Gheorghiciuc for organizing the program. The third author was supported by NSF MSPRF grant DMS-2202827.

\section{Preliminaries}\label{sec: prelim}

A graph is a structure $G=(V,E)$ where $V$ is a set of vertices and $E\subseteq V^2$ is an irreflexive symmetric relation (so, we consider only simple graphs). An edge is a pair $\{x,y\}$ so that $E(x,y)$. An ordered pair $(x,y)$ in $E$ is called a directed edge.

We abuse notation slightly and write $\Z^2$ for the square grid, i.e.~the Cayley graph of $\Z^2$ with respect to its usual generating set:
\[\Z^2=(\Z^2, \{(\vec x, \vec y) :\vec x-\vec y\in S\}\] where $S=\{(1,0),(-1,0),(0,1),(0,-1)\}.$

If $\approx$ is an equivalence relation on the vertex set $V$, we write $[v]_\approx$ for the equivalence class of $v$, and we suppress the subscript when $\approx$ is clear. The quotient graph by $\approx$ is defined by
\[(G/\approx):= \;(V/\approx, \{([x],[y]): E(x,y),  x\not\approx y\}).\] 

For $A\subseteq V$, we write $N(A)$ for the set of neighbors of elements of $A$: \[N(A)=\{v\in V: (\exists u\in A)\;E(u,v)\}.\] We write $N(v)=N(\{v\})$. The graph metric on $V$ is 
\[d(x,y)=\min\{n: y\in N^n(X)\}.\] This is an extended metric as $d(x,y)$ may be infinite if there is no path from $x$ to $y$. The connectedness relation is the equivalence relation $E_G(x,y):\lra d(x,y)<\infty$. For $A, B\subseteq V$, $d(A,B)=\min\{d(a,b): a\in A, b\in B\}$.

A ray in a graph is an injective function $r: \N\rightarrow V$ so that $E(r(n),r(n+1))$ for all $n$.

For $A\subseteq V$, $\partial A$ is the external vertex boundary of $A$. That is, $\partial A=N(A)\setminus A$. This is the smallest set of vertices needed to disconnected $A$ from the rest of the graph.

 A combinatorial embedding of a graph is an assignment for each vertex $v$ of a cyclic order to the edges incident to $v$ (elsewhere, these are called rotation systems). Given a drawing of a graph on an orientable surface, we can get a combinatorial embedding by taking the edges around a vertex in clockwise order. For connected locally finite graphs, all combinatorial embeddings arise this way. \cite[Chapter 3]{MoharThomassen}

 We equip $\Z^2$ it's usual combinatorial embedding:
 \[(x, (1,0)\cdot x), (x, (0,1)\cdot x), (x, (-1,0)\cdot x), (x, (0,-1)\cdot x)\] are in counterclockwise order. If $G$ has a combinatorial embedding and $E$ is an equivalence relation on $G$ with finite connected classes, then $G/E$ inherits a combinatorial embedding.
 
 A facial path in a combinatorial embedding is a sequence of edges and vertices \[... e_i,v_i, e_{i+1},v_{i+1}, e_{i+2}, ...\] so that, for all $i$, $v_i$ is incident to $e_i$ and $e_{i+1}$, and $e_{i+1}$ is the successor of $e_i$ in the cyclic order around $v_i$. An oriented face is a maximal facial path so that no edge appears twice unless its vertices appear in reverse order (this last case is to cover for bridges). Note that an oriented face may be infinite (and $\Z$-indexed). A face boundary is the set of vertices and edges in an oriented face. 

 For this paper, a face is an equivalence class of oriented faces. For oriented faces $x,y$, define 
\[x\equiv_F y\;:\lra \; (\exists n\in \Z) \;x(i)=y(i+n)\] where the indices may be taken modulo the length of $x$. Then a face is an $\equiv_F$-class. In most cases, a face can be identified with its face boundary. The one exception is if $G$ is a bi-infinite path or a cycle. Then, $G$ has two faces with the same boundary, and these are distinguished by their orientation. 
 
 If $G$ is finite and its combinatorial embedding comes from a drawing of $G$ on a surface, a face boundary is a minimal set needed to topologically disconnect a region that doesn't contain any vertices or edges. A triangulation is a graph equipped with a combinatorial embedding where every face boundary has 3 edges.
 
For a set $A$ of face boundaries in $G$, $\partial_F A=\bigcup A\cap V$ is the set of vertices in any of the faces in $A$. For a set of vertices $A\subseteq V$, the topological boundary of $A$, $\partial_\tau A$, is the set of all vertices not in $A$ that share a face with $A$. This is the smallest set of vertices needed to topologically disconnect $A$ from the rest of the graph.

 A graph is planar if it can be drawn in the plane with disjoint arcs representing the edges. For connected locally finite graphs, this is equivalent to having a planar embedding satisfying Euler's formula ($v-e+f=2$) on finite subgraphs. \cite{MoharThomassen} 
 
 We say that a graph is Borel if its vertex set is a standard Borel space and its edge set is Borel in the product Borel structure. We say that a graph is Borel planar if it is Borel and it has a Borel combinatorial embedding that satisfies Euler's formula. There are other notions of Borel planarity in the literature. Our is slightly stronger than some. A thorough discussion is given in an appendix.

A graph $\s G=(V,E)$ where $V$ is a standard Borel space is said to preserve a measure $\mu$ on $V$ if, for any partial Borel bijection $f\subseteq E$ (i.e.~any set of directed edges with no two incident to a common vertex), $\mu(\dom(f))=\mu(\im(f))$. For locally countable graphs, this is equivalent to saying that $\s G$ is the Schreier graph of a measure preserving group action, or that any (partial) Borel bijection $f: V\rightarrow V$ so that $f(v)$ is connected to $v$ for all $v$ is measure preserving. We call a graph $\s G$ equipped with a probability measure that it preserves a pmp (probability measure preserving) graph. When we speak of properties of a pmp graph, we mean properties that hold when we restrict to a co-null set. So, a pmp graph has a measurable $k$-coloring if there is a measurable labelling that almost surely gives each vertex a different label from its neighbors.

For a pmp graph $\s G$, we write $\chi_{\mu}(\s G)$ for the smallest number of colors in a measurable coloring of $\s G$.

For a graph $G$, $G_p$ is the random graph where each edge is deleted from $G$ independently with probability $(1-p).$ There are two threshold probabilities attached to $G_p$ when $G$ is connected, which we denote $p_c(G)$ and $p_u(G)$. For $p<p_c(G)$, $G_p$ almost surely has only finite components, and for $p>p_c(G)$, $G_p$ almost surely has at least one infinite component. For $p>p_u(G)$, $G_p$ almost surely has exactly one infinite component and for $p<p_u(G)$, $G_p$ has either infinitely many or no infinite components. For square grids, $p_c(\Z^2)=p_u(\Z^2)=1/2.$ \cite{HaggstromSurvey}

Let us emphasize that $G_p$ is a random object. When we say that $G_p$ has a property $\phi$, either $\phi$ is a property of probability distributions or we mean that $\phi$ holds almost surely.

We write $[0,1]^G$ for the space of vertex-and-edge labellings of $G$ and $[0,1]^V$ for the space of vertex labellings. For any set $A$, The group $\aut(G)$ acts on $A^G$ and $A^V$ by
\[\rho\cdot x=x\circ \rho\inv.\] We equip $[0,1]^G$ with the product measure $\lambda^G$, where $\lambda$ is Lebesgue measure.

We write $\aut(G)$ for the group of automorphisms of $G$. If $G$ is equipped with combinatorial embedding, we mean the group of orientation preserving automorphisms. So, if $e$ is clockwise of $f$ around $v$, $\rho(e)$ is clockwise of $\rho(f)$ around $\rho(v)$. For example, $\aut(\Z^2)$ is generated by translations and rotations, but does not contain reflections. And, we write $\stab(x)$ for the stabilizer of $x$ under $\aut(G)$:
\[\stab(x)=\{\rho\in \aut(G): \rho(x)=x\}.\]

A factor map is a measurable map from $F:A^G\rightarrow B^G$ so that, for any $\rho\in \aut(G)$, 
\[\rho\cdot F(x)=F(\rho\cdot x).\] A factor of iid process on $G$ is a random $A$-labelling of $G$ which is the pushforward of a factor map $F: [0,1]^G\rightarrow A^G$. (These are restricted definitions, but they suffice for our purposes). Note that $[0,1]^G$ and $[0,1]^V$ factor onto each other, so it doesn't matter which we choose as our source variables. 

A graph $G$ is transitive if $\aut(G)$ acts transitively on vertices. We say $G$ is unimodular if $|\stab(x)\cdot y|=|\stab(y)\cdot x|$ for all $x,y\in V$. And, for $G$ transitive, we say that $G$ is amenable if, for any $\epsilon>0$ there is some $A\subseteq V$ with $|\partial A|/|A|<\epsilon$

\section{Percolation quotients} \label{sec: perc quotient}

This paper mostly concerns amenable planar graphs, and grids in particular. But, we will briefly record a few generalities about percolation on unimodular transitive graphs. Most of what we say here will generalize to unimodular random graphs, but we'll stick to the transitive case for the sake of exposition.

\begin{dfn}
    For a connected countable unimodular graph $\s G=(V,E)$, ${G}_p$ is the random subgraph of $G$ given by Bernoulli bond percolation, so each edge is included in $G_p$ independently with probability $p$. We call $G_p$ the \textbf{percolation subgraph} of $G$. Given $v\in V$, write $[v]_p$ for the (random) $G_p$-component of $v$. We write $[v]$ when $p$ is clear from context.

    The \textbf{percolation quotient} of $G$ is the graph $\widetilde G_p$ obtained by taking the quotient of $G$ by the $G_p$-connectedness relation. So, 
    \[\widetilde G_p=\bigl(\{[v]_p: v\in V\}, \{([v]_p,[u]_p): (v,u)\in E, [v]_p\not=[u]_p\}\bigr)\]
\end{dfn}

When $p<p_c(G)$, the Aizenmann--Barsky theorem says that $\bb P(\bigl|[v]\bigr|\geq n)$ decays exponentially \cite{aizenmannbarsky}. So, $\widetilde G_p$ is nearly quasi-isometric to $G$: there is logarithmic distortion between $G$ and $\widetilde G_p$ rather than constant.

In the critical and supercritical phases, the geometry of $\widetilde G_p$ is more mysterious. For instance, it is unclear whether or not $(\Z^3)_p$ has any (injective) rays when $p>p_c$.


We are interested in factor of iid labellings of $\widetilde G_p$. There are a few ways we might make this precise which turn out to be equivalent. We will use the following definition. 

\begin{dfn}
    An \textbf{fiid labelling of} $\widetilde G_p$ is a factor of iid labelling of $G$ which is almost surely constant on $G_p$-classes. 
\end{dfn}

You might recall the correspondence between $\aut(G)$-fiid labellings of $G$ and measurable labellings of the Bernoulli graphing over $G$, $\s B(G)$. A similar correspondence is available in our setting when $G_p$ has at most one infinite component. 

\begin{dfn}
Let $G=(V,E)$ be a transitive unimodular graph with a fixed root $v$. The \textbf{Bernoulli graphing} over $G$ is the pmp graph $\s B(G)=([0,1]^G/R, F)$ where $R$ is the stabilizer of the root in $\aut(G)$
\[R=\{\rho\in \aut(G): \rho(v)=v\},\] and 
\[F(R\cdot x, R\cdot y)\;:\lra\;(\exists \gamma\in \aut(G))\;\gamma\cdot x=y, \mbox{ and }E(\gamma(v), v). \]

We equip $\s B(G)$ with the measure
\[\mu(A)=\lambda^G(\{x: R\cdot x\in A\}),\] where $\lambda$ is Lebesgue measure.
\end{dfn}

In other words, $\s B(G)$ is the space of rooted isomorphism classes of labellings of $G$, and two labellings are adjacent if you can get one from the other by moving the root to a neighbor. At first glance, $\s B(G)$ may appear quite wild, but with probability $1$, a labelling $x\in [0,1]^G$ assigns unique values to each vertex. It follows that every component of $\s B(G)$ is isomorphic to $G$. Further $R$ is compact, so $[0,1]^G/R$ is a standard Borel space. By our conventions for $\aut(G)$, if $G$ has a combinatorial embedding, then $\s B(G)$ has a Borel combinatorial embedding.

Our main interest in $\s B(G)$ is that $\aut(G)$-equivariant factors of iid variables on $G$ correspond to measurable labellings of $\s B(G)$. If $f:[0,1]^G\rightarrow k^G$ is an fiid labelling of $G$, then $F([x])=f(x)(v)$ defines a measurable labelling of $\s B(G)$ and every labelling of $\s B(G)$ arises in this manner. In particular, $G_p$ can be viewed as a measurable subgraph of $\s B(G)$. 

\begin{dfn}
For $p\in[0,1]$, we write $\omega_p$ for the subgraph of $\s B(G_p)$ where $(R\cdot x,R\cdot y)\in \omega_p$ if and only if, whenever $\gamma\cdot x=y$, $x(v,\gamma(v))<p$. (Note: this doesn't depend on the choice $\gamma,x,$ or $y$)

If $p$ is so that $G_p$ almost surely has finite components, write $E_p$ for the connectedness relation in $\omega_p$. We define
\[\s B(\widetilde G_p)=\s B(G)/E_p\] Note that, since $E_p$ is smooth on a co-null set, $\s B(\widetilde G_p)$ carries a standard Borel structure. We define a measure $\widetilde\mu$ on $\s B(\widetilde G_p)$ by
\[\widetilde \mu(A)=\frac{1}{\bb E(\bigl|[x]\bigr|\inv)} \int_{x\in A} \bigl|[x]\bigr|\inv \;d\mu.\]

If $p$ is such that there is almost surely a unique connected component, $\s B(\widetilde G_p)$ is defined similarly, but first we restrict to the induced subgraph on vertices with finite $\omega_p$-component.
\end{dfn}

If you prefer the language of unimodular random rooted graphs, we could work with the random rooted graph $(\widetilde G_p, [v])$, but we need to re-weight the distribution to make it unimodular. Again, observe that if $G$ is planar, then $\s B(\widetilde G_p)$ is Borel planar for all $p$.

The next proposition gives an analogous correspondence between labelling $\widetilde G_p$ and labelling $\s B(\widetilde G_p)$, and it confirms that $\widetilde \mu$ is preserved by $\s B(\widetilde G_p)$.

\begin{prop}

Fix a countable transitive group $G$. 
\begin{enumerate}
    \item The graph $\s B(\widetilde G_p)$ preserves the measure $\widetilde \mu.$ 
    \item For $p<p_c$ $f: \s B(\widetilde G_p)\rightarrow k$ measurable, 
\[F: [0,1]^G\rightarrow k^G\]
  \[F(x)(u)=f([R\cdot \gamma\cdot x])\] where $\gamma(u)=v$ defines an $\aut(G)$-factor map which is $E_{G_p}$-invariant. Further, any $E_{G_p}$-invariant factor map arises in this way.

    \item For $p>p_u$, the same formula defines a factor map into the space of partial $k$-labellings of $G$ whose domain includes all finite components of $G_p$.
\end{enumerate}

\end{prop}
\begin{proof}
First we check that $\s B(\widetilde G_p)$ prserves $\widetilde \mu$. Fix a partial bijection $f\subseteq \s B(\widetilde G_p)$. We want to show that $\widetilde \mu(\dom(f))=\widetilde \mu(\im(f)).$ Without loss of generality, we may assume that the components in $\dom(f)$ all have the same size $n$, and the components in $\im(f)$ all have the same size $m$ (recall that vertices in $\s B(\widetilde G_p)$ are $\omega_p$ components). For each $C\in \dom(f)\cup \im(f)$, we can choose some $x_C\in C$. If $f(C)=D$, then $x_C$ and $x_D$ are connected in $\s B(G)$, and $x_C\mapsto x_{f(C)}$ defines a bijection between $A=\{x_C: C\in\dom(f)\}$ and $B=\{x_C: C\in \im(f)\}$. Thus, \[\frac{1}{n}\mu(\{x: [x]\in \dom(f)\})=\mu(A)=\mu(B)=\frac{1}m\mu\{x: [x]\in \im(f)\}.\] So, by definition of $\widetilde \mu$, $\widetilde\mu(\dom(f))=\widetilde\mu(\im(f))$ as desired.

For items $(2)$ and $(3)$, labellings of $\s B(\widetilde G_p)$ correspond to $E_{p}$-invariant labellings of $\s B(G)$. By the labelling correspondence for $\s B(G)$, these correspond to $E_{G_p}$-invariant factor maps for $G$.
\end{proof}

\begin{cor}
    $\chi_{\mu}(\s B(\widetilde G_p))=\chi_{fiid}(\widetilde G_p)$.
\end{cor}

When $G$ is amenable, hyperfiniteness of $\s B(G)$ passes to $\s B(\widetilde G_p)$. This gives us some easy upper bounds on chromatic number.

\begin{lem} \label{lem: hyperfinite}
    Suppose $G$ is amenable, then $\s B(G)$ is hyperfinite.
\end{lem}
\begin{proof}
    In fact, the quotient of a hyperfinite equivalence relation by a smooth subequivalence relation is always hyperfinite: If $E=\bigcup E_n$ where each $E_n$ has finite classes and $F\subseteq E$ has a selector $s$, then $[x] E'_n [y]:\lra s(x) E_n s(y)$ has finite classes and $\bigcup E'_n = (E/F)$. 
\end{proof}
\begin{cor}
    If $G$ is amenable, $\chi_{fiid}(\widetilde G_p)\leq 7$.
\end{cor}
\begin{proof}
    By Conley--Miller \cite{ConleyMiller}, hyperfiniteness implies $\chi_{\mu}(\s B(\widetilde G_p))\leq 2\chi(\s B(\widetilde G_p)).$ By the four-color theorem, this is at most $7$.
\end{proof}

\section{Grids at and above criticality} \label{sec: supercritical}

 In the critical and supercritical regime, the Russo--Seymour--Welsh theorem gives us a great deal of structural information about percolation on $\Z^2$. By a rectangle in $\Z^2$ we mean a product of discrete intervals, $R=([a,b]\times [c,d])\cap \Z^2$. The aspect ratio of a rectangle is the ratio of its side lengths, $(c-d)/(a-b)$ in our example.

\begin{thm}[\cite{Russo}\cite{SeymourWelsh}]
    Suppose $p\geq 1/2$. For each $r\in \Q$, there is some probability $p_r>0$ so that, for any rectangle $R$ in $\Z^2$ with aspect ratio $r$, the probability there is a path from the left edge to the right edge of $R$ in $\Z^2_p$ is at least $p_r$.
\end{thm}

So, when there is an infinite component in $(\Z^2)_p$, our coloring problem trivializes.

\begin{cor}
    Fix $p>1/2$, let $\s C$ be the infinite component of $(\bb Z^2)_p$. Then, $\widetilde {(\Z^2)}_p\setminus\{\s C\}$ has finite components. In particular, $\chi_{fiid}\bigl((\widetilde{\Z^2})_2\bigr)\leq 4.$
\end{cor}
\begin{proof}
    By considering nested $3^n\times 3^n$ squares, we can partition $\Z^2$ into rings built up of rectangles with aspect ratio $1/3$. Each of these rectangles is traversed by a path in $(\Z^2)_p$ with some positive probability $p_{1/3}$. So, by the FKG inequality, each ring contains a closed looped with probability at least $p_{1/3}^4$. Since the rings are disjoint, these probabilities are independent. Thus, each vertex in $\Z^2$ is, almost surely, surrounded by infinitely many rings. One such ring must intersect $\s C$. So, each vertex in $\Z^2$ is either in $\s C$, or its component in $\Z^2\setminus \s C$ is bounded in some ring and thus finite.
    
    Now, to color $(\widetilde{\Z^2})_p$, we can choose a color for $\s C$ then color each of the finite components of $(\widetilde{\Z^2})_p\setminus\{C\}$ separately.
\end{proof}

It is unclear how far this argument generalizes. 

\begin{dfn}
    For a countable graph $G$, define
    \begin{align*} p_t(G)&= \inf\{p>p_u: \widetilde G_p \mbox{ has an infinite ray}\} \\
    &=\inf\{p>p_u: G\setminus \s C_p \mbox{ has finite components}\}\end{align*} where $\s C_p$ is the infinite component in $G_p$.
\end{dfn}

\begin{qst} 
    Is $p_t(G)=p_u(G)$ for all $G$? For all planar amenable $G$? For $\Z^3$?
\end{qst}

It is instructive to consider the critical case. Here, each component is finite, so the Russo--Seymour--Welsh theorem says that the percolation quotient almost surely has a structure like a thickened one-ended tree. Let us make this precise.

\begin{dfn}
    For a graph $G=(V,E)$, say that a finite subset $A\subseteq V$ is \textbf{surrounded} by $B\subseteq V$ if every infinite ray from $A$ passes through $B$. Write $A\sqsubseteq B$ in this case.
\end{dfn}

The Russo--Seymour--Welsh theorem implies every percolation component is surrounded by another percolation component when $p\geq 1/2$. 

\begin{prop}
    The relation $\sqsubseteq$ is a partial order on $(\Z^2)_{1/2}$-components, and for any $v$, $C_v:=\{[u]: [v]\sqsubset [u]\}$ is a chain with a minimal element.
\end{prop}
\begin{proof}
    Reflexivity and transitivity are clear. Suppose $A\sqsubset B$, i.e.~$A\not= B$ and $A\sqsubseteq B$. There is an infinite ray from $A$ which must pass through $B$, and since $B$ is finite the ray must leave $B$ at some step. Consider the tail of the ray at this last step in $B$. This tail doesn't pass through $B$ again, so it must not pass through $A$. Thus $B\not\sqsubseteq A$, which gives anti-symmetry. 

    Russo--Seymour--Welsh (and the finiteness of components) implies $C_v\not=\emptyset$. Suppose towards contradiction that $C_v$ contains incomparable elements $A, B$. Then, since $A, B$ are disjoint, any ray from $v$ must meet one of $A$ or $B$ first. Without loss of generality, say $A$. Taking an initial segment, we get a path $r$ from $[v]$ to $A$ that avoids $B$. Since $A\not\sqsubseteq B$, we can find a ray $s$ from $A$ that does not pass through $B$. Since $A$ is connected, we can join up $r$ and $s$ to get a ray showing that $[v]\not\sqsubseteq B$, a contradiction. Finally, since each element of $C$ is finite, it can only surround finitely many components. Thus, $C$ has a smallest element.
\end{proof}

From here, one can use techniques for $3$-coloring one-ended trees to $5$-color $(\Z^2)_{1/2}$. This bound is subsumed by later results, so we won't pursue that path here.

\section{Amenable graphs} \label{sec: amenable}

We know that if $G$ is amenable, $\chi_{fiid}(\widetilde G_p)\leq 7$. Our goal is to shave off a few colors. In fact, our proof will work for all amenable pmp graphs.

Our strategy for coloring $\widetilde {(\Z^2)}_p$ in the subcritical case, and for coloring amenable planar graphs more generally, is familiar. We borrow deep results from the theory of finite surface graph coloring and combine them with a specialized toast construction. We will first explain the finitary theorem we require and how to use it to get Borel colorings from an appropriate toast. Then, we will explain how to build the required toast in the next section. 

\begin{dfn} \label{dfn: restricted}
    For a planar graph $G=(V, E)$ and list of face boundaries $C_1,..., C_n$ and vertices $v_i\in C_i$, a list assignment $L: V\rightarrow \s P(A)$ is \textbf{restricted on $(\vec C, \vec v)$} if
    \begin{enumerate}
        \item For all $i$, $L(v_i)\not=\emptyset$
        \item $|L(u)|\geq 3$ for $u\not\in \{v_1,..., v_n\}$
        \item $|L(v)|\geq 5$ for $u\not\in \bigcup_i C$.
    \end{enumerate}

\end{dfn}

Though not in this language, restricted assignments were introduced by Thomassen as part of an ingenious inductive hypothesis in a proof that planar graphs are list 5-colorable. Postle and Thomas generalized Thomassens's theorem.

\begin{thm}[\cite{PostleThomas}] \label{thm: postlethomas}
    There is a distance $D\in \N$ so that, for any planar graph $G$ with face boundaries $C_1,..., C_n$ and vertices $v_i\in C_i$, if the $C_i$ are pairwise at distance at least $D$ and $L$ is a list assignment which is restricted on each $C_i$, then $G$ admits an $L$-coloring. 
\end{thm}

To apply the Postle--Thomas theorem in the measurable setting, we first need a toast whose boundaries look like face boundaries.

\begin{dfn}
    For $f: \N\rightarrow \R$, an $f$-\textbf{toast} on a graph $G=(V,E)$ is a family $\s T\subseteq [V]^{<\infty}$ with the following properties:
    \begin{itemize}
        \item Complete: $\bigcup \s T=V$
        \item Laminar: For any $A, B\in \s T$, $A\cap B$ is one of $A, B,$ or $\emptyset$
        \item $n$-separated: For any $A, B$, $d(\partial A, \partial B)\geq n$
    \end{itemize}

    We say that a toast $\s T$ on a planar graph $G$ is \textbf{topological} if each $A\in \s T$ is connected and $\s T$ is
    
    \begin{itemize}
        \item $(\tau, n)$-separated: For any $A, B\in \s T$, $d(\partial_{\tau} A,\partial_{\tau} B)\geq n.$
    \end{itemize}

\end{dfn}

\begin{thm}
    For $n\in \N$ large enough, if $\s G$ is a Borel planar graph with a Borel topological $n$-toast $\s T$ so that $\partial_{\tau} A$ is $2$-regular for $A\in \s T$, then $\s G$ is Borel $5$-colorable.
\end{thm}
\begin{proof} Set $n=2D+2$ with $D$ as in the Postle--Thomas theorem.

    For each $A\in \s T$, $\partial_{\tau} A$ is a cycle, possibly of odd length. Choose (in a Borel way) a vertex $v_a\in \partial A$ and let $f$ be a $2$-coloring of $\bigcup_{A\in \s T} \partial_\tau A\setminus v_A$ which uses colors $0$ and $1$. 

    Consider the list assignment on \[A':=A\setminus \bigcup\{ B\cup(\partial_\tau B\setminus v_B): B\subset A\}\] that says vertices adjacent to toast boundaries can't use $0$ or $1$ and special vertices on toast boundaries must use $2$. More formally:
    \[L(u)=\left\{ \begin{array} {cc}\{2\}  & (\exists B\in \s T)\;u=v_B \\ \{2,3,4\}  & u\in \bigcup_{B\in \s T}N(\partial_\tau B)\setminus v_B \\
    \{0,1,2,3,4\} & \mbox{otherwise}\end{array}  \right.\]

   Since $\s T$ is a topological toast, the vertices in $A$ adjacent to some toast boundary form a face boundary in $G\res A'$. And since $\s T$ is $(2D+2)$-separated, these facial boundaries are $D$-separated. Thus, by the Postle--Thomas theorem, $A'$ admits an $L$-coloring. Choosing an $L$-coloring on each piece of toast extends $f$ to a coloring of $\s G$.

\end{proof}

\section{Building topological toast} \label{sec: toast}

Now we'll build our toast. A key lemma is that planar pmp graphs can be extended to triangulations. 

\begin{prop}
    If $T$ is a triangulation, $\partial_\tau S=\partial S$ for all $S\subseteq V$.
\end{prop}

\begin{lem}\label{lem: triangulation}
    If $\s G$ is a hyperfinite, pmp, planar triangulation, then, for any $n$, $\s G$ admits a measurable topological $n$-toast $\s T$ so that $\partial_\tau A$ is $2$-regular for $A\in \s T$. 
\end{lem}
\begin{proof}We may assume that $\s G$ has no finite components.

    By Bowen--Kun--Sabok \cite{BowenKunSabok}, $\s G$ admits a measurable $2n+5$-toast where each piece is connected if every component of $\s G$ is one-ended (note that this is weaker than their notion of connected toast). If $\s G$ is 2-ended, work of Conley and Miller \cite{ConleyMiller} shows that $\s G$ has a measurable toast $\s R_0$, and we can find a $2n+5$-separated family of cuts of the graph. We can then get a toast with connected pieces by considering the maximal regions between cuts contained in pieces from $\s R_0$. So, let $\s R$ be a toast on $\s G$ with connected pieces. Since $\s G$ is a triangulation, $\partial_\tau(A)= \partial A$ for $A\subseteq \s G$. We just need to arrange for boundaries to be $2$-regular.

    Say that a vertex $v$ is penned in by $A\in \s R$ if $\{v\}\sqsubseteq \partial A\cup A\setminus\{v\}$, i.e.~any ray starting from $v$ passes through $\partial A\cup A\setminus \{v\}$. In particular, if $v\in \partial A$ and so are all of the neighbors of $v$, then $v$ is penned in by $A$.
        
    Set $P(A)=A\cup\{v: v\mbox{ is penned in by }A\},$ and consider $\s R'=\{P(A): A\in \s R'\}.$ Notice that $P(A)$ is finite and connected for any $A\in \s R$. Suppose towards contradiction that $v\in \partial P(A)\setminus \partial A$, then there is some neighbor $u$ of $v$ that is penned in by $A$, but $v$ is not in $A$ and $v$ is not penned in by $A$. But this means a ray from $u$ can escape via $v$, which is a contradiction. So, $\partial P(A)\subseteq \partial A$. It follows that no vertex outside $P(A)$ is penned in by $P(A)$ and that $\s R'$ is also $(\tau, 2n+3)$-separated. Let us check that $\s R'$ is laminar. If $A\subseteq B$ then $P(A)\subseteq P(B)$. So, suppose that $A\cap B=\emptyset$ and $A,B\in \s R$. Since $A$ is connected, if any vertex of $A$ is not penned in by $B$, then no vertex of is penned in by $B$. In this case $P(A)\cap P(B)=\emptyset$. And, if every vertex of $A$ is penned in by $B$, then $P(A)\subseteq P(B)$. So, $\s R'$ is a topological toast.
    
    If $v\in \partial A\in \s R'$, then $v$ is not penned in by $A$, so $v$ is adjacent to some vertex not in $\partial A$. In particular, the induced graph on $\partial A$ has all degrees at least $2$ (if $e$ is adjacent to $v\in \partial A$ and $e\not\in \partial A$, there are vertices in $\partial A$ clockwise and counterclockwise from $e$, which must be different since $\s G$ is a triangulation). I claim that every vertex in $\partial A$ has degree exactly $2$ in $\s G\res \partial A$.

    Fix any $v\in \partial A$. The neighbors of $v$ in $A$ (or rather the edges connected to them) must be consecutive in the cyclic order around $v$. Otherwise, $A$ would have to pen in a vertex. And, since $\s G$ is a triangulation, $v$ has neighbors $x,y\in \partial A$ which come just before and after the block of $A$ neighbors. If $v$ had any other neighbors in $\partial A$, then $A$ would have to pen in one of $x$ or $y$, so $v$ has exactly two neighbors in $\partial A$, as desired, and $\s R'$ is our required toast.
\end{proof}

\begin{lem}
    If $\s G$ is a locally finite Borel planar graph, there is a locally finite planar triangulation $\s H$ extending $\s G$. Further if $\s G$ is hyperfinite or pmp, so is $\s H$. 
\end{lem}

\begin{proof}
    Say $\s G$ has edge-set $E$. Since being hyperfinite or pmp only depend on $E_G$, it suffices to find a locally finite triangulation with edge set $T$ so that $E\subseteq T\subseteq E_G$.

    Any finite face with vertices $v_1,..., v_n$, we can triangulate by adding edges $\{v_1,v_i\}$ for $i=2,...,n$ (after picking an order using the Lusin--Novikov theorem). To handle the infinite faces, we want to pick a toast on each face. Here's how we can do that formally.
    
    Let $\s X$ be the space of pairs $(v, F)$ where $F$ is an infinite face and $v$ is a vertex on $F$. Recall that, for us, a face is an equivalence class of oriented faces, so we can set
    \[\s X=\{(v,[x]): x(0)=v\}.\] Put a graph on $\s X$ by declaring $(v, F)$ and $(u, C)$ adjacent when $u,v$ are adjacent in $\s G$ and $F=C$. The space of oriented faces is a Borel subset of $\s G^\Z$, so it carries a standard Borel structure. Each oriented face only visits a vertex finitely many times, so $\s X$ is a quotient of the space of oriented faces along a finite-class equivalence relation. And, since $\equiv_F$ is hyperfinite (as it is induced by a $\Z$-action), so is the graph on $\s X$. Therefore, we can find a toast structure $\s F$ on $\s X$ where each piece has is connected.

    We will attach edges to pieces of toast $A\in \s F$ by induction on the number of pieces of toast contained in $A$. Since each piece of toast is connected, all the vertices $(v, F)\in A$ share a common face $F$. Let $F_i$ be the infinite face after we have added edges at stage $i$ with $F_i\subseteq F$. Say that a vertex is exposed at stage $i$ if it is in $F_i$. Note that the orientation of $F$ gives us an order on the vertices. We will ensure that after adding edges to $A$, $A$ is triangulated and only the first and last vertex in $A$ are exposed. Suppose we have triangulated every piece of toast contained in $A$, and the vertices are $v_0,v_1,..., v_n$ in order. Add edges in clockwise order from $v_0$ to $v_j=1,...,n$ across $F_i$. Since $A$ is connected, the result is still planar an only $v_0$ and $v_n$ are left exposed since the edge between disconnects every other vertex from the infinite face. After $\omega$-many steps, no vertices are left exposed. Carrying out this process on every face simultaneously triangulates $\s G$. 
\end{proof}

Combining our lemmas, we get the main theorem:

\begin{thm}
    Any hyperfinite, locally finite, pmp, planar graph $\s G$ is measurably $5$-colorable.
\end{thm}
\begin{proof}
    We can extend $\s G$ to a triangulation. The triangulation has arbitrarily well separated topological toast with 2-regular boundaries, so it has a measurable $5$-coloring. This restricts to a $5$-coloring of $\s G$.
\end{proof}
\begin{cor}
    For all $p$ $\chi_{fiid}((\widetilde\Z^2)_p\leq 5$.
\end{cor}

\begin{qst}
    Is there a hyperfinite planar pmp graph with no measurable $4$-coloring? Is $\s B((\widetilde \Z^2)_p)$ measurably 4-colorable for $p\leq 1/2$?
\end{qst}

\bibliographystyle{plain}
\bibliography{refs}

\appendix

\section{Borel planarity}

In the body of this article, we used combinatorial embeddings to encode planarity. The main utility of this approach was that let us identify faces in a Borel way and orient each face. The other main approach to planarity in the descriptive setting is via 2-bases. This is the approach taken by Conley--Gaboriau--Marks--Tucker-Drob, for instance \cite{ConleyMarksGaboriauTuckerDrob}. 

\begin{dfn}
    A \textbf{2-basis} for a graph is a basis $\s B$ for the cycle space consisting of simple cycles so that each edge is contained in at most 2 elements of $\s B$.
\end{dfn}

If $G$ is 2-connected and finite, a 2-basis is equivalent to the set of faces of a planar drawing of $G$.

\begin{dfn}
    A \textbf{b-planar} Borel graph is a Borel graph with a Borel 2-basis.
\end{dfn}

Being b-planar is a weaker condition than being planar. Given a Borel planar graph, we can identify the faces and select a 2-basis (with perhaps some difficulty about bridges). But, if we only have a 2-basis, we can't identify  consistent orientation to cyclically order the edges around each vertex. 

For example, let's consider $\Z^2$ as a pure graph (no combinatorial embedding). Then, $\aut(\Z^2)$ contains both orientation preserving and orientation reversing automorphisms and preserves the 2-basis of commuting squares. The quotient of the Bernoulli shift
\[[0,1]^{\s Z^2}/\stab((0,0))\] is b-planar but not planar.

However, all of our arguments still go through for b-planar graphs at the expense of some essentially notational headaches. After we pass to triangulations, we only use the fact that our graphs are locally planar (and that $\partial_\tau$ and $\partial$ agree for triangulations). In the proof that Borel planar graphs extend to triangulations, we use our combinatorial embeddings first to identity the space of oriented faces (in such a way that $\equiv_F$ is a hyperfinite relation), and then to order the vertices on each face. This can be avoided as follows. 

\begin{prop}
    If $\s G$ is a locally finite b-planar Borel graph, then $\s G$ extends to a locally finite b-planar triangulation.
\end{prop}

\begin{proof}
    We can decompose $\s G$ into $2$-connected blocks and cut vertices. Contract each finite 2-connected block into a single vertex. If we can triangulate this contracted graph in a Borel way, then we can triangulate the original graph.
    
    If an edge $e$ is in a 2-connected block, it lies on an infinite face if and only if it is in only one element of the 2-basis for $\s G$. If $e$ is not in a 2-connected block, either one of its endpoints is a cut vertex or it connects two 2-connected blocks along infinite faces (since we have contracted the finite blocks). Any choice of cyclic order of edges around the cut vertices extends to a planar embedding. So, we can fix an cyclic order arbitrarily at each cut vertex, then say that an oriented infinite face is a bi-infinite path of edges that either lie on infinite faces of 2-connected blocks, cross between infinite faces of 2-connected blocks, or that follow the chosen order at each cut vertex.

    Now define a face as an equivalence class of oriented faces by the relation
    \[x\equiv_F y\;:\lra \;(\exists \rho\in \aut(\Z))\;x\circ \rho=y. \] This is still a hyperfinite equivalence relation, and for any vertex $v$ there are still only finitely many orientations of a fixed face with the origin at $v$. So, as in the proof of \ref{lem: triangulation}, we can form a hyperfinite graph $\s X$ of pairs of faces and vertices.

     There is no longer a natural global order on the space $\s X$, rather, the vertices on a face can come in two possible orders. This is still sufficient because each piece of toast can choose an order independently. The rest of the proof follows exactly as in Lemma \ref{lem: triangulation}.

\end{proof}

\end{document}